\documentclass[runningheads]{llncs}
\usepackage[T1]{fontenc}
\usepackage{amsfonts}
\usepackage[dvipsnames]{xcolor}
\usepackage{bm}
\usepackage{amsmath}
\usepackage{amssymb}
\usepackage{url}
\usepackage{comment}
\usepackage[hidelinks]{hyperref}
\usepackage{wrapfig,lipsum,booktabs}
\usepackage{etoolbox}
\usepackage{graphicx}
\usepackage{bbold}
\usepackage{cleveref}
\usepackage{xfrac}

\usepackage{enumerate}
\usepackage{verbatim}
\usepackage{bbm}
\usepackage[T1]{fontenc}
\usepackage{graphicx}
\newcommand{\Tr}{\operatorname{Tr}}
\newcommand{\todo}[1]
  {\vspace{5 mm}\par \noindent 
  \framebox{\begin
  {minipage}[c]{0.95 \textwidth} \tt #1
\end{minipage}}\vspace{5 mm}\par}
\begin{document}
\title{Infinite-dimensional Siegel disc as symplectic and K\"ahler quotient}
%
%
\author{Alice Barbara Tumpach\inst{1, 2}\orcidID{0000-0002-7771-6758} }
\authorrunning{A. B. Tumpach}

\institute{Wolfgang Pauli Institut, Oskar-Morgensternplatz 1, 1090 Vienna, Austria\\
 \and
Laboratoire Painlev\'e, Lille University, 59650 Villeneuve d'Ascq, France
\email{alice-barbora.tumpach@univ-lille.fr}\\
\url{http://www.geometricgreenlearning.com}
 }
\maketitle              
\begin{abstract}
In this paper, we construct the restricted infinite-dimensional Siegel disc as a Marsden-Weinstein symplectic reduced space and as K\"ahler quotient of a weak K\"ahler manifold. The obtained symplectic form is invariant with respect to the left action of the infinite-dimensional restricted symplectic group and coincides with the Kirillov-Kostant-Souriau symplectic form of the restricted Siegel disc obtained via the identification with an affine coadjoint orbit of the restricted symplectic group, or equivalently with a coadjoint orbit of the universal central extension of the restricted symplectic group.

\end{abstract}
%
%
%
\section{Introduction}

\subsection{Finite-dimensional picture.}
Let us recall the theory of finite-dimensional dual pairs \cite{Skerritt_Vizman}.
Consider the set $\operatorname{Mat}_{2n\times 2n}(\mathbb{R})$ of real $2n\times2n$ matrices  endowed with the complex structure
$
J_n = \begin{pmatrix}
0 & I_n\\
-I_n &0\end{pmatrix},
$
where $I_n$ denotes the identity matrix of size $n\times n$. Together with the scalar product
$
\langle u, v \rangle = \Tr u^T v, 
$
where $u, v \in \operatorname{Mat}_{2n\times 2n}(\mathbb{R}),$ the vector space 
$\operatorname{Mat}_{2n\times 2n}(\mathbb{R})$ is a K\"ahler flat space with symplectic form defined for $u, v \in \operatorname{Mat}_{2n\times 2n}(\mathbb{R})$ by
$
\Omega(u, v) = \langle u, J v \rangle = \Tr u^T J v.
$
Since the group $\operatorname{GL}(2n, \mathbb{R})$ of invertible $2n\times 2n$ matrices is an open set in $\operatorname{Mat}_{2n\times 2n}(\mathbb{R}),$ the K\"ahler structure of $\operatorname{Mat}_{2n\times 2n}(\mathbb{R})$ restricts to a K\"ahler structure on $\operatorname{GL}(2n, \mathbb{R})$. Note that this K\"ahler structure is not invariant by left or right translations by elements of the group. Consider the following subgroups of $\operatorname{GL}(2n, \mathbb{R})$:
\[
\operatorname{Sp}(2n, \mathbb{R}) = \{ a \in \operatorname{GL}(2n, \mathbb{R}), a^T J a = J\}
\textrm{ and }
\operatorname{O}(2n, \mathbb{R}) = \{ a \in \operatorname{GL}(2n, \mathbb{R}), a^T a = I_{2n}\}.
\]
Let $\operatorname{Sp}(2n, \mathbb{R})$ (resp. $\operatorname{O}(2n, \mathbb{R})$) act on the left (resp. right) on $\operatorname{GL}(2n, \mathbb{R})$ (more generally on $\operatorname{Mat}_{2n\times 2n}(\mathbb{R})$) by matrix multiplication.  Note that both actions commute and admit equivariant moment maps:

\setlength{\unitlength}{0.4 mm}
\begin{picture}(100, 45)
\put(130,40){$\operatorname{GL}(2n, \mathbb{R})$}
\put(130,25){$(\langle\cdot, \cdot \rangle, J, \Omega)$}
\put(125,30){\vector(-3, -2){20}}
\put(175,30){\vector(3, -2){20}}
\put(100, 30){$\mu_{\operatorname{Sp}}$}
\put(185, 30){$\mu_{\operatorname{O}}$}
\put(60, 10){$\mathfrak{sp}(2n, \mathbb{R})^*\simeq \mathfrak{sp}(2n, \mathbb{R})$}
\put(200, 10){$\mathfrak{o}(2n, \mathbb{R})^*\simeq \mathfrak{o}(2n, \mathbb{R})$}
\end{picture}
\\
where, for $M \in \operatorname{GL}(2n, \mathbb{R})$,
$
\mu_{\operatorname{Sp}}(M) = -\frac{1}{2}M M^TJ \textrm{  and  } \mu_{\operatorname{O}}(M) = -\frac{1}{2}M^TJ M,
$
with the identifications $\mathfrak{sp}(2n, \mathbb{R})^*\simeq \mathfrak{sp}(2n, \mathbb{R})$ and $\mathfrak{o}(2n, \mathbb{R})^*\simeq \mathfrak{o}(2n, \mathbb{R})$ given by (the real part of)  the trace.
It was show in \cite[Proposition 4.2]{Skerritt_Vizman} (see also \cite{Libermann_Marle,Balleier_Wurzbacher}) that each level set of the momentum map for the $\operatorname{Sp}(2n, \mathbb{R})$-action is a homogeneous space under $\operatorname{O}(2n, \mathbb{R})$, and vice-versa, each level set of the momentum map for the $\operatorname{O}(2n, \mathbb{R})$-action is a homogeneous space under $\operatorname{Sp}(2n, \mathbb{R})$. This property is refered to as \textit{mutually transitive actions}, and the groups $\operatorname{Sp}(2n, \mathbb{R})$ and $\operatorname{O}(2n, \mathbb{R})$ form a so-called \textit{dual pair}. It was proved in \cite[Proposition 2.8]{Skerritt_Vizman} that, given a pair of mutually transitive actions on a symplectic manifold $M$, the symplectic reduction with respect to one group is symplectomorphic to a coadjoint orbit of the dual group. In particular, the symplectic quotient at $J\in \mathfrak{o}(2n, \mathbb{R})$
\[
\mu^{-1}_{\operatorname{O}}(J)/\operatorname{O}(2n, \mathbb{R})_J 
\]
is symplectomorphic to a coadjoint orbit of $\operatorname{Sp}(2n, \mathbb{R})$. In fact, the level set $\mu^{-1}_{\operatorname{O}}(J)$ is simply $\operatorname{Sp}(2n, \mathbb{R})$
and the stabilizer $\operatorname{O}(2n, \mathbb{R})_J$ of $J$ under the $\operatorname{O}(2n, \mathbb{R})$-action  is easily seen to be isomorphic to the unitary group $\operatorname{U}(n)$. The quotient space is therefore $\operatorname{Sp}(2n, \mathbb{R})/\operatorname{U}(n)$ and is isomorphic to the Siegel disc \cite[Section~4]{Nielsen}, \cite{Siegel,Barbaresco}:
\[
\mathfrak{D}(n) = \{Z \in \operatorname{Mat}_{n\times n}(\mathbb{C})\mid  Z^T = Z\; \text{and}\;
\operatorname{id}_{\mathcal{H}_-}-Z^* Z\succ 0\}.
\]
This symplectic reduction was used in \cite{Ohwasa} in order to relate  two different formulations of the Gaussian wave packet dynamics commonly used in semiclassical
mechanics. 
In line with the finite-dimensional theory,  $\mathfrak{D}(n)$ can be realized as a $\operatorname{Sp}(2n, \mathbb{R})$ (co-)adjoint orbit, namely as the (co-)adjoint orbit of $J\in \mathfrak{sp}(2n, \mathbb{R})$.

Another example of the theory of dual pairs is the pair  $\left(\operatorname{U}(n), \operatorname{U}(m)\right)$ acting by left and right multiplications on $\operatorname{Mat}_{n\times m}(\mathbb{C})$. The realization of the Grassmannian $\operatorname{Gr}(m, n)$ as a symplectic quotient of $\operatorname{U}(m)$ is a crucial step in the geometric proof of the connectedness of the space of constant norm Parseval frames of $n$ vectors in $\mathbb{C}^m$ know in Signal Processing and Frame Theory as the Frame Homotopy Conjecture \cite{DykemaStrawn,Needham_Shonkwiler}.

\subsection{Infinite-dimensional picture}

In the infinite dimensional setting, examples of dual pairs are know notably in fluid dynamics 
\cite{Gay-Balmaz_Vizman1,Gay-Balmaz_Vizman2}
and in the theory of diffeomorphisms groups and their co-adjoint orbits \cite{Haller_Vizman1,Haller_Vizman2,Haller_Vizman3,Blaga_Salazar_Tortorella_Vizman}. 
In the context of operator theory, we have shown in \cite{Tum_Gr,Tum_phd} that (each connected component of the) restricted Grassmannian is a symplectic and K\"ahler quotient of a flat space of operators by the right action of a unitary group modeled on  trace class operators. A. Weinstein  then raised the question of whether the restricted Grassmannain could be realized as a coadjoint orbit of a unitary group. The answer has been given in \cite{BRT} and is twofold: while the restricted Grassmannian is not a coadjoint orbit of the restricted unitary group, it is an affine coadjoint orbit or, equivalently, a coadjoint orbit of the universal central extension of the restricted unitary group. Consequently,  the restricted Grassmannian is an infinite-dimensional example that can be realized as symplectic quotient of one group and at the same time as the coadjoint orbit of a ``dual'' group. However, as far as we know, the general theory of dual pairs in this context has not been fully investigated.

The study of the infinite-dimensional counterpart of the Siegel disc is motivated on one hand by Teichm\"uller theory \cite{Tum_Teichmuller} and on the other hand by recent developments of the theory of Gaussian processes in infinite dimensions \cite{Minh1,Minh2,Minh3}. In particular, the (real) restricted Siegel disc appears in \cite{Minh1} as the parameter space of  zero-mean Gaussian measures on an infinite-dimensional Hilbert space which are equivalent to a fixed non-degenerate Gaussian measure \cite[Equ.~17]{Minh1}.
In \cite{GRT_proceedings} it was shown that the restricted Siegel disc is a coadjoint orbit of the universal central extension of the restricted symplectic group (see Section~\ref{section_Spres} for the definition of this group). In the present paper, we give the mirror presentation of the restricted Siegel disc as symplectic and K\"ahler quotient. 

\subsection{Contributions}

The present paper contains the following contributions:
\begin{itemize}
\item[$\bullet$] In Theorem~\ref{disc restreint de Siegel = espace homogene}, we prove that the proper subgroup $\operatorname{Sp}_{1,2}(\mathcal{V}, \Omega)\subset \operatorname{Sp}_{\textrm{res}}(\mathcal{V}, \Omega)$ acts transitively on the restricted Siegel disc. 
Note that the restricted Siegel disc is a subset of the set of Hilbert-Schmidt operators, whereas $\operatorname{Sp}_{1,2}(\mathcal{V}, \Omega)$ is build using both  both trace class and Hilbert-Schmidt operators.
\item[$\bullet$] In Theorem~\ref{momentum map}, we compute the momentum map of the right action of the orthogonal group $\operatorname{O}_{1,2}(\mathcal{V})$ on a flat K\"ahler space $\mathcal{M}_{1,2}(\mathcal{V})$, and we prove 
in Theorem~\ref{non-equi} that this momentum map is not equivariant.
\item[$\bullet$] We show in Propositions~\ref{prop1} and \ref{prop2} that the right action of the unitary group $\operatorname{U}_1(\mathcal{H}_+)$ preserves the K\"ahler structure of  $\mathcal{M}_{1,2}(\mathcal{V})$, as well as the zero level set of the momentum map.
\item[$\bullet$] In Theorem~\ref{quotient}, we show that the K\"ahler reduced space at $0$ is isomorphic to the restricted Siegel disc, and  we compute explicitly in Theorem~\ref{symplectic_form} the resulting symplectic form.
\end{itemize}
\section{Mathematical background}

\subsubsection{The general 
linear group $\operatorname{GL}(\mathcal{V})\subset \operatorname{GL}(\mathcal{H})$.}
Consider a real infinite-dimensional separable Hilbert space 
$\mathcal{V}$ with inner product 
$\langle\cdot, \cdot\rangle_{\mathcal{V}}$.
Endow $\mathcal{V}$ with a complex structure $J$ compatible with 
$\langle\cdot, \cdot\rangle_{\mathcal{V}}$ (this is always possible in the infinite-dimensional case):
$J^2=-\operatorname{id}_\mathcal{V}\;\;\text{and}\;\;
\langle Ju,Jv\rangle_{\mathcal{V}}=\langle u,v\rangle_{\mathcal{V}},
$
and denote by $\Omega$ the strong symplectic bilinear form  
$
\Omega(u,v)=\langle u,Jv\rangle_{\mathcal{V}}.$
The real infinite-dimensional Hilbert space $\mathcal{V}$ endowed with $\big(\langle\cdot, \cdot\rangle_{\mathcal{V}}, J, \Omega\big)$ is a (flat) K\"ahler space.

Let $\operatorname{GL}(\mathcal{V})$ be the general 
linear group of real linear bounded invertible operators
on $\mathcal{V}$. Recall that $\operatorname{GL}(\mathcal{V})$ 
is a Banach Lie group whose Banach Lie algebra 
$\mathfrak{gl}(\mathcal{V}):= L(\mathcal{V})$ is the Banach space  of 
all real linear bounded  operators on $\mathcal{V}$.
Consider the complexification $\mathcal{H} := 
\mathcal{V}\otimes \mathbb{C}$ of $\mathcal{V}$ with hermitian 
scalar product
$
\langle u, v\rangle_\mathcal{H} := 
\langle \Re u, \Re v\rangle_{\mathcal{V}} + 
\langle \Im u, \Im v\rangle_{\mathcal{V}} + 
i \langle \Im u, \Re v\rangle_{\mathcal{V}} - 
i \langle \Re u, \Im v\rangle_{\mathcal{V}},
$
and extend $J$ and $\Omega$ by $\mathbb{C}$-linearity to $\mathcal{H}$ 
and decompose $\mathcal{H}$ as
$
\mathcal{H}=\mathcal{H}_+\oplus\mathcal{H}_-,$ 
where $\mathcal{H}_\pm:=\operatorname{Eig}_{\pm i}(J)$ denotes the eigenspace of 
$J$ associated to $\pm i$. With respect to this decomposition,
the operators 
$a\in \mathfrak{gl}(\mathcal{V})$ that preserve the real subspace $\mathcal{V}\subset \mathcal{H}$, extended complex linearly 
to $\mathcal{H}$,  have the block form
$ a=\begin{pmatrix} g&h\\ \bar h&\bar g
\end{pmatrix}. $

\subsubsection{The general 
linear group $\operatorname{GL}_{1,2}(\mathcal{V})\subset \operatorname{GL}(\mathcal{V})$.}

The group $\operatorname{GL}_{1,2}(\mathcal{V})$ is defined as the following group of  invertible operators
\[
\operatorname{GL}_{1,2}(\mathcal{V}) = \left\{ \begin{pmatrix} g&h\\ \bar h&\bar g
\end{pmatrix}\in \operatorname{GL}(\mathcal{H}), h\in L^2(\mathcal{H}_-,\mathcal{H}_+),
g - \operatorname{id}_{\mathcal{H}_+} \in L^1(\mathcal{H}_+)
\right\},
\]
where $L^1(\mathcal{H}_+)$ denotes the  Banach space of the trace class operators on $\mathcal{H}_+$, and $L^2(\mathcal{H}_-, \mathcal{H}_+)$  the Hilbert space of Hilbert-Schmidt operators from $\mathcal{H}_-$ to  $\mathcal{H}_+$.
It is a real Banach Lie group modeled on the real Banach Lie algebra 
\[
\mathfrak{gl}_{1,2}(\mathcal{V}) = \left\{ \begin{pmatrix}
A_1            &A_2\\
\bar{A}_2 &\bar{A}_1\end{pmatrix}\in \mathfrak{gl}(\mathcal{H}), A_1\in L^1(\mathcal{H}_+), A_2\in L^2(\mathcal{H}_-,\mathcal{H}_+)\right\}
\]
endowed with the norm 
$
\left\|\begin{pmatrix}
A_1            &A_2\\
\bar{A}_2 &\bar{A}_1\end{pmatrix}\right\|_{1,2} = \|A_1\|_1+ \|A_2\|_2,
$
where $\|\cdot\|_1$ denotes the Trace class norm and $\|\cdot\|_2$ denotes the Hilbert-Schmidt norm.

\subsubsection{The symplectic linear group $\operatorname{Sp}(\mathcal{V},\Omega)\subset \operatorname{GL}(\mathcal{V})$.}

The expression of the complex bilinear form $\Omega$ with respect 
to the  hermitian scalar product 
$\langle\cdot, \cdot\rangle_\mathcal{H}$ reads
$
\label{complex_omega_j_inner_product_relation}
\Omega(u, v) = \langle u, \overline{J v}\rangle_\mathcal{H} 
= \langle J v, \overline{u}\rangle_\mathcal{H}.
$
The \textbf{symplectic linear group} $\operatorname{Sp}(\mathcal{V},\Omega) $ is the subgroup of $\operatorname{GL}(\mathcal{V})$ which preserves the symplectic bilinear form.
Note that
$a\in\operatorname{Sp}(\mathcal{V},\Omega)$ if and only if 
$a\in\operatorname{GL}(\mathcal{V})$ and $a^* J a = J$.
Using the block decomposition with respect to the direct sum $\mathcal{H} = \mathcal{H}_+\oplus 
\mathcal{H}_-$, we conclude that $a = \begin{pmatrix} g&h\\ \bar h&\bar g
\end{pmatrix}$ belongs to $\operatorname{Sp}(\mathcal{V},\Omega)$ iff
$g^*g-h^T\bar{h}=\operatorname{id}_{\mathcal{H}_+}$ and 
$g^*h=h^T\bar{g}\in L(\mathcal{H}_-, \mathcal{H}_+),$
where the transpose is defined by
$
h^T:=(\bar{h})^*  \in L(\mathcal{H}_-, \mathcal{H}_+).
$

\subsubsection{The  restricted symplectic linear group 
${\rm Sp}_{\rm res}(\mathcal{V},\Omega)$.}\label{section_Spres}
The restricted
symplectic group is
\begin{align}
\operatorname{Sp}_{\rm{res}}(\mathcal{V},\Omega)
&=\left\{\left.
\begin{pmatrix}
g            &h\\
\bar{h} &\bar{g}\end{pmatrix}
\in\operatorname{Sp}(\mathcal{V},\Omega) \,\right|\, 
h\in L^2(\mathcal{H}_-,\mathcal{H}_+)\right\}.
\end{align}

\subsubsection{The symplectic linear group $\operatorname{Sp}_{1,2}(\mathcal{V},\Omega)$.}

The Banach Lie group $\operatorname{Sp}_{1,2}(\mathcal{V},\Omega)$ is defined as  \begin{align}
\operatorname{Sp}_{1,2}(\mathcal{V},\Omega)
&=\left\{\left.
\begin{pmatrix}
g            &h\\
\bar{h} &\bar{g}\end{pmatrix}
\in\operatorname{Sp}(\mathcal{V},\Omega) \,\right|\, 
h\in L^2(\mathcal{H}_-,\mathcal{H}_+), g - \operatorname{id}_{\mathcal{H}_+} \in L^1(\mathcal{H}_+)\right\}.
\end{align}


\subsubsection{The orthogonal group  $\operatorname{O}(\mathcal{V})$.}
On the real Hilbert space $\mathcal{V}$ with scalar product $\langle\cdot, \cdot\rangle_{\mathcal{V}}$ acts the orthogonal group $\operatorname{O}(\mathcal{V})$ consisting of operators $a\in \operatorname{GL}(\mathcal{V})$ such that 
\[
\langle a u, a v\rangle_{\mathcal{V}} = \langle  u,  v\rangle_{\mathcal{V}}.
\]
Note  that  $a \in \operatorname{O}(\mathcal{V})$ iff  $a^T = a$ where the transpose $T$ denotes the adjoint with respect to the scalar product $\langle\cdot, \cdot\rangle_{\mathcal{V}}$. The $\mathbb{C}$-linear extension of $\operatorname{O}(\mathcal{V})$ to $\mathcal{H} = \mathcal{V}\oplus \mathbb{C}$ (also denoted by  $\operatorname{O}(\mathcal{V})$ ) preserves $\mathcal{V}$ hence consists of operators $a$ such that 
\[
a (\Re u + i \Im u) = a \Re u + i a\Im u = \Re (a u) + i \Im (a u).
\]
Consequently $a \Re u = \Re a u$ and $a \Im u = \Im a u$ and 
\begin{align*}
\langle a u, a v\rangle_\mathcal{H} &:= 
\langle \Re a u, \Re a v\rangle_{\mathcal{V}} + 
\langle \Im a u, \Im a v\rangle_{\mathcal{V}} + 
i \langle \Im a u, \Re a v\rangle_{\mathcal{V}} - 
i \langle \Re a u, \Im a v\rangle_{\mathcal{V}},\\
&= 
\langle a \Re  u, a\Re  v\rangle_{\mathcal{V}} + 
\langle a\Im  u, a\Im  v\rangle_{\mathcal{V}} + 
i \langle a\Im  u, a\Re  v\rangle_{\mathcal{V}} - 
i \langle a\Re  u, a\Im  v\rangle_{\mathcal{V}},\\
& = \langle  u,  v\rangle_\mathcal{H} 
\end{align*}
Therefore, 
\begin{equation}\label{O_condition}
a \in \operatorname{O}(\mathcal{V}) \Longleftrightarrow 
\begin{cases} 
a \in \operatorname{GL}(\mathcal{V})\subset 
\operatorname{GL}(\mathcal{H})\\ 
a^* a = \textrm{id}_{\mathcal{H}}.
\end{cases}
\end{equation}
\subsubsection{The orthogonal group  $\operatorname{O}_{1,2}(\mathcal{V})$.}
The orthogonal group  $\operatorname{O}_{1,2}(\mathcal{V})$ is defined as
\begin{equation}\label{O12}
\operatorname{O}_{1,2}(\mathcal{V}) = \left\{
a = \begin{pmatrix}
g       &h\\
\bar{h} &\bar{g}\end{pmatrix} \in \operatorname{GL}_{1,2}(\mathcal{H}), a^* a = \textrm{id}_{\mathcal{H}}
\right\}.
\end{equation}
It is a Banach Lie group with Lie algebra
\begin{equation}\label{o12}
\mathfrak{o}_{1,2}(\mathcal{V}) = \left\{
 A = \begin{pmatrix}
A_1       &A_2\\
\bar{A_2} &\bar{A_1}\end{pmatrix} \in \mathfrak{gl}_{1,2}(\mathcal{H}), A^*+  A = 0
\right\}.
\end{equation}

\subsubsection{The infinite-dimensional restricted Siegel disc $\mathfrak{D}_{\rm{res}}(\mathcal{H})$.}
The infinite-dimensional  restricted Siegel disc is, by definition,
\[
\mathfrak{D}_{\rm{res}}(\mathcal{H}):=
\{Z\in L^2(\mathcal{H}_-,\mathcal{H}_+)
\mid  Z^T = Z\; \text{and}\;
\operatorname{id}_{\mathcal{H}_-}-Z^* Z\succ 0\},
\]
where, for $A\in L(\mathcal{H}_+,\mathcal{H}_+)$ the notation
$A\succ 0$ means $\langle Au,u\rangle_\mathcal{H}>0$, for all
$u\in\mathcal{H}_+, u\neq 0$, and where
$Z^T=(\bar{Z})^* \in L(\mathcal{H}_-,\mathcal{H}_+)$ with $\bar{Z}u := \overline{Z\bar{u}}$, $u \in \mathcal{H}_+$.
The infinite-dimensional  restricted Siegel disc is an open set in the complex Hilbert space
\begin{equation}\label{def_E_star}
\mathcal{E}_{\rm{res}}(\mathcal{H}):=
\{V\in L^2(\mathcal{H}_-,\mathcal{H}_+)\mid V^T = V\}.
\end{equation}

\section{$\mathfrak{D}_{\rm res}(\mathcal{H})$ as Homogeneous space under $\operatorname{Sp}_{1,2}(\mathcal{V},\Omega)$.}
It was shown in  \cite[Theorem~1]{GRT_proceedings} that $\operatorname{Sp}_{\rm res}(\mathcal{V},\Omega)$ acts transitivelly on $\mathfrak{D}_{\rm res}(\mathcal{H})$.
In fact, the restriction of the $\operatorname{Sp}_{\rm res}(\mathcal{V},\Omega)$-action to the proper subgroup  $\operatorname{Sp}_{1,2}(\mathcal{V},\Omega) \subsetneq \operatorname{Sp}_{\rm res}(\mathcal{V},\Omega)$ is still  transitive on $\mathfrak{D}_{\rm res}(\mathcal{H})$. 
\begin{remark}
In the present paper we use Bourbaki convention for quotient spaces: $G/H$ is the set of equivalence classes of the form $gH$, $g\in G$, where $H$ acts on the right on $G$. This differs from \cite{GRT_proceedings} but is coherent with the notation used for symplectic and K\"ahler quotients \cite{Tum_Gr}.
\end{remark}
\begin{theorem}
\label{disc restreint de Siegel = espace homogene} 
The map $\rho: \operatorname{Sp}_{1,2}(\mathcal{V},\Omega)
\times\mathfrak{D}_{\rm res}(\mathcal{H})\longrightarrow
\mathfrak{D}_{\rm res}(\mathcal{H}),$
\begin{equation}\label{Sp_action}
\begin{aligned}
 &\left(
\begin{pmatrix}
g            &h\\
\bar{h} &\bar{g}\end{pmatrix}
,Z\right)\longmapsto (gZ+h)(\bar{h}Z+\bar{g})^{-1}
\end{aligned}
\end{equation}
is a well-defined transitive action of the symplectic 
group $\operatorname{Sp}_{1,2}(\mathcal{V},\Omega)$ onto the restricted Siegel disc. The isotropy group of $0$ 
is isomorphic to the unitary group
\[
\operatorname{Sp}_{1,2}(\mathcal{V},\Omega)_0\simeq\operatorname{U}_1(\mathcal{H}_+) := \operatorname{U}(\mathcal{H}_+)\cap \{\operatorname{id}_{\mathcal{H}_+} + L^1(\mathcal{H}_+)\}.
\]
Therefore, the orbit map induces a diffeomorphism of real Banach manifolds
\begin{equation}\label{diffeo1}
\operatorname{Sp}_{1, 2}(\mathcal{V},\Omega)/\operatorname{U}_1(\mathcal{H}_+)\longrightarrow \mathfrak{D}_{\rm res}(\mathcal{H}),\quad 
\left[
\begin{pmatrix}
g       &h\\
\bar{h} &\bar{g}\end{pmatrix}
\right]_{\operatorname{U}(\mathcal{H}_+)}\longmapsto h\bar{g}^{-1}.
\end{equation}
\end{theorem}

\begin{proof} 
Let us show that the action is transitive. A direct computation shows that
\[
\exp \begin{pmatrix}
0       &A\\
A^* &0\end{pmatrix} = \begin{pmatrix}
\cosh |A|      & A \frac{\sinh |A|}{|A|}\\
A^*\frac{\sinh |A^*|}{|A^*|} &\cosh |A^*| \end{pmatrix}, 
\]
where, for $A \in L^2(\mathcal{H}_-,\mathcal{H}_+)$, $|A| = \left(A^*A\right)^{\frac{1}{2}}$ and $|A^*| = \left(A A^*\right)^{\frac{1}{2}}$.
In particular, for $A \in L^2(\mathcal{H}_-,\mathcal{H}_+)$ such that $A^T = A$, one has $A^* = \bar{A}$ and 
\[
\exp \begin{pmatrix}
0       &A\\
\bar{A} &0\end{pmatrix}= \begin{pmatrix}
\cosh |A|      & A \frac{\sinh |A|}{|A|}\\
 \bar{A} \frac{\sinh |\bar{A}|}{|\bar{A}|}  &\cosh |\bar{A}| \end{pmatrix}\in \operatorname{Sp}_{1,2}(\mathcal{V},\Omega),
\]
since $ A \frac{\sinh |A|}{|A|}\in L^2(\mathcal{H}_-,\mathcal{H}_+)$ and $\cosh |A| = \sum_{n = 0}^\infty \frac{|A|^{2n}}{(2n)!} \in \operatorname{id}_{\mathcal{H}_+} + L^1(\mathcal{H}_+)$.
\[
\exp \begin{pmatrix}
0       &A\\
\bar{A} &0\end{pmatrix}\cdot 0 = \begin{pmatrix}
\cosh |A|      & A\frac{\sinh |A|}{|A|}\\
\frac{\sinh |A|}{|A|}\bar{A}  &\cosh |A| \end{pmatrix}\cdot 0 = A \frac{\sinh |A|}{|A|}\left(\cosh |A|\right)^{-1} = A |A|^{-1} \tanh |A|.
\]
For $Z\in L^2(\mathcal{H}_-,\mathcal{H}_+)$ such that 
$\operatorname{id}_{\mathcal{H}_-}-Z^*Z\succ 0$, define
\begin{equation}\label{action_transitive}
A_Z:=Z\frac{\operatorname{argtanh}|Z|}{|Z|}=
Z\sum_{k=0}^\infty\frac{1}{2k+1}(Z^*Z)^k\in 
L^2(\mathcal{H}_-,\mathcal{H}_+).
\end{equation}
When $Z^T=Z$ we have $A_Z^T=A_Z$. One has 
\[
|A_Z|^2  = A_Z^* A_Z = \frac{\operatorname{argtanh}|Z|}{|Z|} Z^*  Z\frac{\operatorname{argtanh}|Z|}{|Z|} = \left(\operatorname{argtanh}|Z|\right)^2,
\]
hence $|A_Z| = \operatorname{argtanh}|Z|$ and
\[
A_Z |A_Z|^{-1} \tanh |A_Z| = Z\frac{\operatorname{argtanh}|Z|}{|Z|}\left(\operatorname{argtanh}|Z|\right)^{-1} \tanh  \left(\operatorname{argtanh}|Z|\right) = Z.
\]
This proves that $
\exp \begin{pmatrix}
0       &A\\
\bar{A} &0\end{pmatrix}\cdot 0  = Z,
$
hence the action of $\operatorname{Sp}_{1,2}(\mathcal{V},\Omega)$ is transitive.
The remainder of the proof is straighforward and follows from the proof of Theorem~1 in \cite{GRT_proceedings}.
\end{proof}

\section{The restricted Siegel Disc as K\"ahler quotient}

\subsection{The group $\operatorname{GL}_{1,2}(\mathcal{V})$ as open set in a weak K\"ahler space} Infinite-dimensional symplectic and K\"ahler quotients in the Banach context were studied in details in \cite{Tum_Gr}. K\"ahler manifolds possess a Riemannian metric, a symplectic form and a compatible formally integrable structure \cite[Definition~2.12]{Tum_Gr}, see also \cite{GLT,GLTNN}. Consider the affine space $\mathcal{M}_{1,2}(\mathcal{V}) := \textrm{id}_{\mathcal{H}} + \mathfrak{gl}_{1,2}(\mathcal{V}).$
\begin{proposition}
The affine space $\mathcal{M}_{1,2}(\mathcal{V}) := \textrm{id}_{\mathcal{H}} + \mathfrak{gl}_{1,2}(\mathcal{V})$ is a weak K\"ahler space with the following scalar product $\textrm{g}_{\mathcal{M}}$, complex structure $J_{\mathcal{M}}$ and symplectic bilinear form $\omega_{\mathcal{M}}$:
\[
\left\{\begin{array}{l}
\textrm{g}_{\mathcal{M}}(V_1, V_2) =  \Tr V_1^* V_2\\
\operatorname{J}_{\mathcal{M}}V_1 = JV_1\\
\omega_{\mathcal{M}}(V_1, V_2) = \Tr V_1^* J V_2.
\end{array}\right.
\]
where the tangent space to $\mathcal{M}_{1,2}(\mathcal{V})$ at any $M$ is identified with $\mathfrak{gl}_{1,2}(\mathcal{V}).$
\end{proposition}

\begin{proof}
Since
$
\begin{pmatrix}
A_1            &A_2\\
\bar{A}_2 &\bar{A}_1\end{pmatrix}^* = \begin{pmatrix}
\bar{A_1}^T            &A_2^T\\
\bar{A}_2^T &A_1^T\end{pmatrix},
$
the Banach space
$\mathfrak{gl}_{1,2}(\mathcal{V})$ is stable by the adjoint $*$ with respect to the Hermitian scalar product $\langle\cdot, \cdot\rangle_{\mathcal{H}}$.  Since it is also stable by product, $V_1^* V_2$ belong to 
$\mathfrak{gl}_{1,2}(\mathcal{V})$
for any $V_1, V_2 \in \mathfrak{gl}_{1,2}(\mathcal{V})$. Therefore $\textrm{g}_{\mathcal{M}}$ takes real values, and is clearly non-degenerate. Moreover, $J$ stabilizes $\mathfrak{gl}_{1,2}(\mathcal{V})$, i.e. $J V_2\in \mathfrak{gl}_{1,2}(\mathcal{V})$ for any $V_2 \in \mathfrak{gl}_{1,2}(\mathcal{V})$. Therefore $\omega_{\mathcal{M}}$ is also real-valued. Consequently  $\Tr (V_1^* J V_2) =  \Tr (V_1^* J V_2)^*  = \Tr V_2^* J^* V_1$. The anti-symmetry then follows from $J^* = -J$. The non-degeneracy of $\omega_{\mathcal{M}}$ is clear.
\end{proof}

\begin{corollary}\label{open}
As open set in $\mathcal{M}_{1,2}(\mathcal{V})$, $\operatorname{GL}_{1,2}(\mathcal{V})$ inherits a flat K\"ahler structure.
\end{corollary}

\begin{remark}
Note that this K\"ahler structure is not invariant by left and right multiplications by arbitrary elements in $\operatorname{GL}_{1,2}(\mathcal{V})$.
\end{remark}

\subsection{Momentum map for the right action of  $\operatorname{O}_{1,2}(\mathcal{V})$ on $\mathcal{M}_{1,2}(\mathcal{V})$}

Recall that a \textit{momentum map} for the action of a Lie group $G$ on the symplectic manifold $\left(\mathcal{M}_{1,2}(\mathcal{V}), \omega_{\mathcal{M}}\right)$ is a function $\mu_{G}$ on $\mathcal{M}_{1,2}(\mathcal{V})$ with values in the dual $\mathfrak{g}^*$ of the Lie algebra $\mathfrak{g}$ of $G$ such that the differential $d\mu_{G}$ of $\mu_{G}$ satisfies
\begin{equation}\label{moment}
\langle d_M\mu_{G}(V) |A\rangle_{G} = \omega_{\mathcal{M}}(X^A(M), V),
\end{equation}
where $M\in \mathcal{M}_{1,2}(\mathcal{V})$,  $V\in T_M\mathcal{M}_{1,2}(\mathcal{V})$ is a tangent vector to $\mathcal{M}_{1,2}(\mathcal{V})$ at $M$, and $X^A$ denotes the vector field generated by $A\in \mathfrak{g}$ on $\mathcal{M}_{1,2}(\mathcal{V})$.
The momentum map of a right action is called \textit{equivariant} when
$\mu_G(M \cdot g) = \operatorname{Ad}^*(g)\left(\mu_{G}(M)\right),$
where $\operatorname{Ad}^*$ denotes the coadjoint action of $G$ on $\mathfrak{g}^*$ defined by $\operatorname{Ad}^*(g)\xi(A) = \xi(\operatorname{Ad}(g)(A))$, $\xi \in \mathfrak{g}^*$, $A\in \mathfrak{g}$.

\begin{theorem}\label{momentum map}
The right action of $\operatorname{O}_{1,2}(\mathcal{V}, \Omega)$ on $\mathcal{M}_{1,2}(\mathcal{V})$ admits an moment map $\mu_{\operatorname{O}}: \mathcal{M}_{1,2}(\mathcal{V})  \longrightarrow  \mathfrak{o}_{1,2}^*(\mathcal{V}, \Omega)$ defined by 
\begin{equation}\label{SP_moment}
 \langle\mu_{\operatorname{O}}(M)| A\rangle_{\operatorname{O}} := -\frac{1}{2}\Tr  \left(M^* J M - J\right) A
\end{equation}
where $M\in \mathcal{M}_{1,2}(\mathcal{V})$, $A\in \mathfrak{o}_{1,2}(\mathcal{V}, \Omega)$, and where $\langle\cdot | \cdot \rangle_{\operatorname{o}}$ denotes the duality pairing between $\mathfrak{o}_{1,2}^*(\mathcal{V})$ and $\mathfrak{o}_{1,2}(\mathcal{V})$.
\end{theorem}

\begin{proof}
First note that, for $M = \textrm{id}_\mathcal{H} + X$ with $X \in \mathfrak{gl}_{1,2}(\mathcal{V})$, one has
\[
M^* J  M - J 
= \left(J + JX + X^* J + X^* J X \right)  - J = JX + X^* J + X^* J X,
\]
hence $M^* J  M - J$ belongs to $\mathfrak{gl}_{1,2}(\mathcal{V})$. Therefore $\left(M^* J  M - J\right) A$ is trace class for any $A\in \mathfrak{o}_{1,2}(\mathcal{V})$. This proves that $\mu_{\operatorname{O}}$ is well-defined. The differential of $\mu_{\operatorname{O}}$ at $M \in \mathcal{M}_{1,2}(\mathcal{V})$ reads
\[
\langle d_M\mu_{\operatorname{O}}(V)| A\rangle_{\operatorname{O}} 
:= -\frac{1}{2}\Tr  \left(V^* J  M + M^* J V\right) A 
= -\frac{1}{2}\left[\Tr  \left(V^* J  M A\right) + \Tr  \left(A M^* J  V \right)\right].
\]
Since $\left(A M^* J  V\right)\in \mathfrak{gl}_{1,2}(\mathcal{V})$, $\Tr \left(A M^* J  V\right)$ is real and equals $\Tr \left(A M^* J  V\right)^*$, hence
\begin{align*}
\langle d_M\mu_{\operatorname{Sp}}(V)| A\rangle_{\operatorname{Sp}} 
&= -\frac{1}{2}\left[\Tr  \left(V^* J  M A \right) + \Tr  \left(V^* J^*  M A^*\right)\right]
= -\frac{1}{2}\left[\Tr  \left(V^* J  M A \right) - \Tr  \left(V^* J  M A^*\right)\right],
\end{align*}
where we have used $J^* = -J$.
Since  for $A \in \mathfrak{o}_{1,2}(\mathcal{V}) $
one has $A^{*} + A = 0$, 
\begin{align*}
\langle d_M\mu_{\operatorname{O}}(V)| A\rangle_{\operatorname{O}} 
& = -\left[ \Tr  \left(V^* J  M A\right)\right]
 = -\omega_{\mathcal{M}}(V, M A) = \omega_{\mathcal{M}}(M A, V).
\end{align*}
\end{proof}
\begin{theorem}\label{non-equi}
The moment map $\mu_{\operatorname{O}}$ is \textit{not} equivariant. One has 
\[
\langle \mu_{\operatorname{O}}(M\cdot a) - \operatorname{Ad}^*(a)\left(\mu_{\operatorname{O}}(M)\right) | A\rangle = 
-\frac{1}{2}\Tr(a^{-1} J a -  J) A
\]
\end{theorem}

\begin{proof}
One has
$
 \langle\mu_{\operatorname{O}}(M\cdot a)| A\rangle_{\operatorname{O}}
  = -\frac{1}{2}\Tr  \left(a^* M^* J M a  - J\right) A, 
$
and
$
 \langle \mu_{\operatorname{O}}(M) | a A a^{-1} \rangle= -\frac{1}{2}\Tr(M^* J  M - J) a A a^{-1}.$
The proposition follows from $a^* =  a^{-1}$, $a\in\operatorname{O}_{1,2}(\mathcal{V}, \Omega)$.
\end{proof}

\subsection{The restricted Siegel disc as  K\"ahler reduced space}
 
 \begin{proposition}\label{prop1}
 The right action of $\operatorname{U}_{1}(\mathcal{H}_+) = \operatorname{O}_{1,2}(\mathcal{V})\cap \operatorname{Sp}_{1,2}(\mathcal{V}, \Omega)$ preserves the K\"ahler structure of $\operatorname{GL}_{1,2}(\mathcal{V})$.
 \end{proposition}

\begin{proof}
For $a \in \operatorname{U}_{1}(\mathcal{H}_+) = \operatorname{O}_{1,2}(\mathcal{V})\cap \operatorname{Sp}_{1,2}(\mathcal{V}, \Omega)$, one has $a^* J  a = J$ and $a^* = a^{-1}$. Hence, for $V_1, V_2 \in \mathfrak{gl}_{1,2}(\mathcal{V})$, one has
\[
\left\{\begin{array}{l}
\textrm{g}_{\mathcal{M}}(V_1\cdot a, V_2 \cdot a) =  \Tr a^* V_1^*  V_2 a = \Tr V_1^*  V_2 =  \textrm{g}_{\mathcal{M}}(V_1, V_2)\\
\operatorname{J}_{\mathcal{M}}(V_1 \cdot a) = J V_1 a = \operatorname{J}_{\mathcal{M}}(V_1)\cdot a\\
\omega_{\mathcal{M}}(V_1 \cdot a, V_2 \cdot a) = \Tr a^*V_1^* J V_2 a = \Tr V_1^* J V_2 = \omega_{\mathcal{M}}(V_1, V_2) .
\end{array}\right.
\]
\end{proof}

\begin{proposition}\label{prop2}
The right action of $\operatorname{U}_{1}(\mathcal{H}_+) = \operatorname{O}_{1,2}(\mathcal{V})\cap \operatorname{Sp}_{1,2}(\mathcal{V}, \Omega)$ preserves the level set $\mu_{\operatorname{O}}^{-1}(0)\subset \operatorname{GL}_{1,2}(\mathcal{V})$.
\end{proposition}

\begin{proof}
It follows from Proposition~\ref{non-equi}, that 
for $a \in \operatorname{U}_{1}(\mathcal{H}_+)= \operatorname{O}_{1,2}(\mathcal{V})\cap \operatorname{Sp}_{1,2}(\mathcal{V}, \Omega)$,
\[
\langle \mu_{\operatorname{O}}(M\cdot a) - \operatorname{Ad}^*(a)\left(\mu_{\operatorname{O}}(M)\right) | A\rangle = 
-\frac{1}{2}\Tr(a^{-1} J a -  J) A = -\frac{1}{2}\Tr(a^{*} J a -  J) A = 0,
\]
for any $A\in \mathfrak{o}_{1,2}(\mathcal{V})$. Hence $\operatorname{U}_{1}(\mathcal{H}_+)$ acts equivariantly on the momentum map. In particular  $\operatorname{U}_{1}(\mathcal{H}_+)$ preserves $\mu_{\operatorname{O}}^{-1}(0)$.
\end{proof}

\begin{theorem}\label{quotient}
The quotient space $\mathcal{Q} := \mu_{\operatorname{O}}^{-1}(0)/\operatorname{U}_{1}(\mathcal{H}_+)$ inherits from  $\operatorname{GL}_{1,2}(\mathcal{V})$ a symplectic and K\"ahler structure and is isomorphic to the restricted Siegel disc:
\[
\mu_{\operatorname{O}}^{-1}(0)/\operatorname{U}_{1}(\mathcal{H}_+)\simeq \operatorname{Sp}_{1,2}(\mathcal{V}, \Omega)/\operatorname{U}_{1}(\mathcal{H}_+) \simeq \mathfrak{D}_{\rm res}(\mathcal{H}) 
\]
\end{theorem}

\begin{proof} 
The $\operatorname{O}_{1,2}(\mathcal{V})$-action on $\mathcal{M}_{1,2}(\mathcal{V})$ preserves the open set $\operatorname{GL}_{1,2}(\mathcal{V})\subset \mathcal{M}_{1,2}(\mathcal{V})$ (Corollary~\ref{open}).
Consequently the momentum map for the $\operatorname{O}_{1,2}(\mathcal{V})$-action on $\mathcal{M}_{1,2}(\mathcal{V})$ restricts to $\operatorname{GL}_{1,2}(\mathcal{V})$. The zero level set $\mu_{\operatorname{O}}^{-1}(0)$ of the restriction of the momentum map to $\operatorname{GL}_{1,2}(\mathcal{V})$ consists of operators  $a \in \operatorname{GL}_{1,2}(\mathcal{V})$ such that $a^* J a = J$, hence $\mu_{\operatorname{O}}^{-1}(0)$ equals  $\operatorname{Sp}_{1,2}(\mathcal{V}, \Omega)$. In particular $\operatorname{Sp}_{1,2}(\mathcal{V}, \Omega)$ acts transitively on the zero level set by left  multiplications. At $\textrm{id}_{\mathcal{H}}\in \operatorname{Sp}_{1,2}(\mathcal{V}, \Omega)  = \mu_{\operatorname{O}}^{-1}(0)$,  the tangent space to the $\operatorname{U}_{1}(\mathcal{H}_+)$-orbit at $\textrm{id}_{\mathcal{H}}$ is the Lie algebra $\mathfrak{u}_1(\mathcal{H}_+)$
and admits a closed complement 
$
\mathfrak{m}=\left\{\left.
\begin{pmatrix}
0       &A\\
\bar{A} &0\end{pmatrix}
\right| A\in L^2(\mathcal{H}_-,\mathcal{H}_+), A^T = A\right\}. 
$
Therefore the quotient space is a smooth manifold \cite[Proposition~11, \S 1.6, Chap.III]{Bo1972}. 
The remainder of the proof is a consequence of \cite[Proposition~2.9]{Tum_Gr} and \cite[Theorem~2.13]{Tum_Gr}. 
\end{proof}

\begin{remark}
Formally integrable complex structures are examples of Nijenhuis operators. The construction of Nijenhuis operators on quotients of Banach Lie groups using conditions on their Lie algebras where studied in \cite{GLTNN}.
\end{remark}

\begin{theorem}\label{symplectic_form}
The resulting symplectic structure on $\mathcal{Q} := \mu_{\operatorname{O}}^{-1}(0)/\operatorname{U}_{1}(\mathcal{H}_+)$ is invariant by the left action of  $\operatorname{Sp}_{1,2}(\mathcal{V}, \Omega)$ and coincide (modulo a multiplicative constant) with the Kirillov-Kostant-Souriau symplectic form of $\mathfrak{D}_{\rm res}(\mathcal{H})$ after identification with an affine coadjoint orbit of $\operatorname{Sp}_{\textrm{res}}(\mathcal{V}, \Omega)$. More precisely, for $\tilde{U}$, $\tilde{V}\in T_{Z}\mathfrak{D}_{\rm res}(\mathcal{H})$ with $Z = h\bar{g}^{-1}$, we have
\begin{equation}
\omega_{\mathcal{Q}}\left(Z\right)\left(\tilde{U}, \tilde{V}   \right)  = 2\Im \Tr \left(\left(\operatorname{id}_{\mathcal{H}_+} -\bar{Z}Z\right)^{-1}\tilde{U}^*\left(\operatorname{id}_{\mathcal{H}_+} -Z\bar{Z}\right)^{-1}\tilde{V}\right).
\end{equation}
\end{theorem}

\begin{proof}
Since $\operatorname{Sp}_{1,2}(\mathcal{V}, \Omega)$ preserves the symplectic form $\omega_{\mathcal{M}}$ of $\mathcal{M}_{1,2}(\mathcal{V})$, the resulting quotient symplectic form $\omega_{\mathcal{Q}}$ is $\operatorname{Sp}_{1,2}(\mathcal{V}, \Omega)$-invariant. At $[\textrm{id}_{\mathcal{H}}]_{\operatorname{U}_1(\mathcal{H}_+)}\in \mathcal{Q}$ the symplectic form $\omega_{\mathcal{Q}}$ reads:
\[
\begin{array}{l}
\omega_{\mathcal{Q}}\left( 
[\operatorname{id}_{\mathcal{H}}]_{\operatorname{U}_1(\mathcal{H}_+)}\right) 
\left(\left[
\begin{pmatrix}
A_1      &A_2\\
\bar{A}_2&\bar{A}_1\end{pmatrix} 
\right]_{ \mathfrak{u}_1( \mathcal{H} _+)},\left[
\begin{pmatrix}
B_1      &B_2\\
\bar{B}_2&\bar{B}_1\end{pmatrix} 
\right]_{ \mathfrak{u}_1( \mathcal{H} _+)}\right)
\\
= \Tr \begin{pmatrix}
0      &A_2\\
\bar{A}_2&0\end{pmatrix}^*
\begin{pmatrix}
i      &0\\
0&-i \end{pmatrix} 
\begin{pmatrix}
0     &B_2\\
\bar{B}_2&0\end{pmatrix} =i \operatorname{Tr}(\bar{A}_2B_2-\bar{B}_2A_2) =   2\Im \Tr(A_2^* B_2).
\end{array}
\]
The left  action of 
$\begin{pmatrix} g&h\\ \bar h&\bar g\end{pmatrix}
\in \operatorname{Sp}_{1,2}(\mathcal{V},\Omega)$ maps  a vector $U\in T_0\mathfrak{D}_{\rm res}(\mathcal{H})$ to 
$\tilde{U} =  g U \bar{g}^{-1} - h \bar{g}^{-1} \bar{h} U \bar{g}^{-1}
 = \left(\operatorname{id}_{\mathcal{H}_+} -Z\bar{Z}\right)g U \bar{g}^{-1} 
= (g^*)^{-1}U \bar{g}^{-1}.$ 
Therefore the expression of the symplectic form at $Z = h\bar{g}^{-1}\in\mathfrak{D}_{\rm{res}}(\mathcal{H})$ reads
\begin{align*}
\omega_{\mathcal{Q}}\left(Z\right)\left(\tilde{U}, \tilde{V}   \right) & = \omega_{\mathcal{Q}}\left(0\right)\left(U, V  \right) 
= 2 \Im \Tr \left(\left(g^*\tilde{U} \bar{g}\right)^*g^*\tilde{V} \bar{g}\right) 
=  - 2 \Im \Tr \left(g g^*\tilde{U}\bar{g} \bar{g}^*\overline{\tilde{V}} \right) \\ & = -2\Im \Tr \left(\left(\operatorname{id}_{\mathcal{H}_+} -Z\bar{Z}\right)^{-1}\tilde{U}\left(\operatorname{id}_{\mathcal{H}_+} -\bar{Z}Z\right)^{-1}\overline{\tilde{V}}\right),
\end{align*}
where we have used 
$
\operatorname{id}_{\mathcal{H}_+} -Z\bar{Z}
=  \operatorname{id}_{\mathcal{H}_+} - (g^*)^{-1} h^T \bar{h} g^{-1} 
= (g^*)^{-1} g^{-1} = (g g^*)^{-1}.
$
Note that since $Z\in\mathfrak{D}_{\rm{res}}(\mathcal{H})$ is a
Hilbert-Schmidt operator, it follows that 
$\operatorname{id}_{\mathcal{H}_+}-Z\bar Z$ is a Fredholm
operator with index zero. Since $\operatorname{id}_{\mathcal{H}_+}-Z\bar Z$ 
is also positive
definite, it is a bijection. 
As shown in \cite[Section~5]{GRT_proceedings}, the resulting symplectic form on $\mathfrak{D}_{\rm res}(\mathcal{H})$ coincide (modulo a multiplicative constant) with the Kirillov-Kostant-Souriau symplectic form of the corresponding affine coadjoint orbit of $\operatorname{Sp}_{\textrm{res}}(\mathcal{V}, \Omega)$. More precisely, $\omega_{\mathcal{Q}} = -\frac{1}{2\gamma}\omega_\gamma$ where $\omega_\gamma$ is the KKS symplectic form of the coadjoint orbit of $(0, \gamma)$ in the universal central extension of $\operatorname{Sp}_{\textrm{res}}(\mathcal{V}, \Omega)$ \cite[Equ. 36]{GRT_proceedings}.
\end{proof}

\begin{remark}
The symplectic form obtained coincides with (-2 times) the imaginary part of the Hermitian metric of the restricted Siegel disc which is a generalization to Hilbert-Schmidt operators of the finite-dimensional formulas, see \cite[Equation 98, section~4]{Nielsen} and \cite{Barbaresco,Barbaresco2}. Note that the complex structure obtained by K\"ahler quotient differs from the homogeneous complex structure of the restricted Siegel disc induced by its injection in $\mathcal{E}_{\rm{res}}(\mathcal{H})$ defined by \eqref{def_E_star}.
\end{remark}

\begin{credits}
\subsubsection{\ackname} This research is funded by FWF under the grant numbers~I-5015-N and PAT1179524. The author would
     like to acknowledge the excellent working conditions and
     interactions at Erwin Schroedinger Institut, Vienna, during the
     thematic programme ``Infinite-dimensional Geometry: Theory and
     Applications'' where part of this work was completed. 
\end{credits}

%
%
%

\end{document}